\documentclass[11pt]{article}

\usepackage{tikz}
\usepackage{graphicx}
\usepackage{verbatim}
\usepackage{url}
\usepackage{fullpage}
\usepackage{amssymb,amsfonts,amsmath,amsthm}
\usepackage[colorlinks=true,linkcolor=blue,citecolor=red]{hyperref}
\usepackage[capitalise]{cleveref}

\newtheorem{corollary}{Corollary}[section]
\newtheorem{maintheorem}[corollary]{Theorem}
\newtheorem{theorem}[corollary]{Theorem}
\newtheorem{proposition}[corollary]{Proposition}
\newtheorem{definition}[corollary]{Definition}
\newtheorem{lemma}[corollary]{Lemma}
\newtheorem{claim}[corollary]{Claim}

\newtheorem{conjecture}[corollary]{Conjecture}

\newtheorem{problem}[corollary]{Problem}

\theoremstyle{definition}

\allowdisplaybreaks
\newcommand{\FormatAuthor}[3]{
\begin{tabular}{c}
#1 \\ {\small\texttt{#2}} \\ {\small #3}
\end{tabular}
}

\newcommand{\defemph}[1]{\textbf{\emph{#1}}}

\newcommand{\colored}{}

\newcommand{\DefineEqual}{\colored{:=}}

\newcommand{\AC}{\mathit{AC}}

\newcommand{\Bits}{\colored{\{0,1\}}}

\newcommand{\Absolute}[1]{\colored{\left|#1\right|}}
\newcommand\abs[1]{\Absolute{#1}}

\newcommand{\HC}{\mathcal{H}}

\newcommand{\ceil}[1]{\colored{{\left\lceil#1\right\rceil}}}

\newcommand{\dist}{\colored{{\sf dist}}}

\newcommand{\N}{{\mathbb N}}

\newcommand{\E}{{\mathbb E}}

\newcommand{\seq}{\subseteq}

\newcommand{\half}{\frac{1}{2}}
\newcommand{\rk}{rk}

\renewcommand{\int}{{\sf int}}

\newcommand{\D}{\mathcal{D}}
\renewcommand{\L}{\mathcal{L}}

\renewcommand{\Pr}{\mathbb{P}}

\newcommand{\maj}{{\sf maj}}

\newcommand{\Bin}{\mathrm{Bin}}

\newcommand{\avgstretch}{{\sf avgStretch}}
\newcommand{\recmaj}[1][k]{{\sf rec\text{-}maj}_{#1}}
\newcommand{\Arecmaj}[1][k]{{A_{\recmaj[#1]}}}
\newcommand{\Zrecmaj}[1][k]{{Z_{\recmaj[#1]}}}
\newcommand{\tribes}{{\sf tribes}}
\newcommand{\Atribes}{{A_{\tribes}}}
\newcommand{\Ztribes}{{Z_{\tribes}}}
\newcommand{\Astribes}{{A^*_{\tribes}}}

\newcommand{\rand}{{\sf rand}}
\newcommand{\mineq}{{\sf \nu}}

\newcommand{\third}{{r}}

\begin{document}

\title{On Mappings on the Hypercube \\ with Small Average Stretch}
\author{
\begin{tabular}[h!]{ccc}
   \FormatAuthor{Lucas Boczkowski}{lucasboczko@gmail.com}{CNRS, IRIF Universit\'{e} Paris 7}
 & \FormatAuthor{Igor Shinkar}{ishinkar@sfu.ca}{Simon Fraser University}
\end{tabular}
} %
\date{}

\maketitle

\begin{abstract}
Let $A \seq \Bits^n$ be a set of size $2^{n-1}$, and let $\phi \colon \Bits^{n-1} \to A$ be a bijection. We define \emph{the average stretch} of $\phi$ as
\begin{equation*}
    \avgstretch(\phi) = \E[\dist(\phi(x),\phi(x'))],
\end{equation*}
where the expectation is taken over uniformly random $x,x' \in \Bits^{n-1}$ that differ in exactly one coordinate.

In this paper we continue the line of research
studying mappings on the discrete hypercube with small average stretch.
We prove the following results.
\begin{itemize}
    \item For any set $A \seq \Bits^n$ of density $1/2$ there exists a bijection $\phi_A \colon \Bits^{n-1} \to A$ such that $\avgstretch(\phi_A) = O(\sqrt{n})$.
    \item For $n = 3^k$ let $\Arecmaj[] = \{x \in \Bits^n : \recmaj[](x) = 1\}$, where $\recmaj[] : \Bits^n \to \Bits$ is the function \emph{recursive majority of 3's}.
    There exists a bijection $\phi_{\recmaj[]} \colon \Bits^{n-1} \to \Arecmaj[]$ such that $\avgstretch(\phi_{\recmaj[]}) = O(1)$.
    \item Let $\Atribes = \{x \in \Bits^n : \tribes(x) = 1\}$.
    There exists a bijection $\phi_{\tribes} \colon \Bits^{n-1} \to \Atribes$ such that $\avgstretch(\phi_{\tribes}) = O(\log(n))$.
\end{itemize}
These results answer the questions raised by Benjamini, Cohen, and Shinkar (Isr. J. Math 2016).
\end{abstract}
\clearpage
{\small \tableofcontents}
\clearpage
\section{Introduction}
\label{sec:introduction}

In this paper we continue the line of research from \cite{BCS, RS18, JS21}
studying geometric similarities between different subsets of the hypercube $\HC_n = \Bits^n$.
Given a set $A \seq \HC_n$ of size $\abs{A} = 2^{n-1}$ and a bijection $\phi \colon \HC_{n-1} \to A$, we define \defemph{the average stretch} of $\phi$ as
\begin{equation*}
    \avgstretch(\phi) = \E_{x \sim x' \in \HC_{n-1}}[\dist(\phi(x),\phi(x'))],
\end{equation*}
where $\dist(x,y)$ is defined as the number of coordinates $i \in [n]$ such that $x_i \neq y_i$, and the expectation is taken over a uniformly random $x,x' \in \HC_{n-1}$ that differ in exactly one coordinate.%
\footnote{Note that any $C$-Lipschitz function $\phi \colon \HC_{n-1} \to A$ satisfies $\avgstretch(\phi) \leq C$. That is, the notion of average stretch is a relaxation of the Lipschitz property.}

The origin of this notion is motivated by the study of the complexity of distributions~\cite{GGN10, Vi12, LV12}.
In this line of research, given a distribution $\D$ on $\HC_n$ the goal is to find a mapping $h \colon \HC_m \to \HC_n$
such that if $U_m$ is the uniform distribution over $\HC_m$, then $h(U_m)$ is (close to) the distribution $\D$,
and each output bit $h_i$ of the function $h$ is computable efficiently (e.g., computable in $\AC_0$, i.e., by polynomial size circuits of constant depth).

Motivated by the goal of proving lower bounds for sampling from the uniform distribution on some set $A \seq \HC_n$, Lovett and Viola~\cite{LV12} suggested the restricted problem of proving that no bijection from $\HC_{n-1}$ to $A$ can be computed in $\AC_0$. Toward this goal they noted that it suffices to prove that any such bijection requires large average stretch. Indeed, by the structural results of~\cite{Hastad86, Bop97, LMN93} it is known that any such mapping $\phi$ that is computable by a polynomial size circuit of depth $d$ has $\avgstretch(\phi) < \log(n)^{O(d)}$, and hence proving that
any bijection requires super-polylogarithmic average stretch implies that it cannot be computed in $\AC_0$.
Proving a lower bound for sampling using this approach remains an open problem.

Studying this problem, \cite{BCS} have shown that if $n$ is odd, and $A_{\maj} \seq \HC_n$ is the hamming ball of density $1/2$, i.e. $A_{\maj} = \{x \in \HC_n : \sum_i x_i > n/2\}$, then there is a $O(1)$-bi-Lipschitz mapping from $\HC_{n-1}$ to $A_{\maj}$, thus suggesting that proving a lower bound for a bijection from $\HC_{n-1}$ to $A_{\maj}$ requires  new ideas beyond the sensitivity-based structural results of~\cite{Hastad86, Bop97, LMN93} mentioned above. In \cite{RS18} it has been shown that if a subset $A_{\rand}$ of density $1/2$ is chosen uniformly at random, then with high probability there is a bijection $\phi \colon \HC_{n-1} \to A_{\rand}$ with $\avgstretch(\phi) = O(1)$.
This result has recently been improved by Johnston and Scott \cite{JS21}, who showed that for a random set $A_{\rand} \seq \HC_n$ of density $1/2$ there exists a $O(1)$-Lipschitz bijection from $\HC_{n-1}$ to $A_{\rand}$ with high probability.

The following problem was posed in \cite{BCS}, and repeated in \cite{RS18,JS21}.

\begin{problem}\label{problem:avg-stretch}
Exhibit a subset $A \subseteq \HC_n$ of density $1/2$ such that
any bijection $\phi \colon \HC_{n-1} \to A$ has $\avgstretch(\phi) = \omega(1)$,
or prove that no such subset exists.%
\footnote{Throughout the paper, the density of a set $A \seq \HC_n$ is defined as $\mu_n(A) = \frac{\abs{A}}{2^n}$.}
\end{problem}

\paragraph{Remark}
  Note that it is easy to construct a set of density $1/2$
  such that any bijection $\phi \colon \HC_{n-1} \to A$ must have a \emph{worst case stretch} at least $n/2$.
  For example, for odd $n$ consider the set $A = \{y \in \HC_n : n/2 < \sum_i y_i < n\} \cup \{0^n\}$.
  Then any bijection $\phi \colon \HC_{n-1} \to A$ must map some point $x \in \HC_{n-1}$ to $0^n$, while all neighbours $x'$ of $x$
  are mapped to some $\phi(x')$ with weight at least $n/2$. Hence, the \emph{worst case stretch} of $\phi$ is at least $n/2$.
  In contrast, \cref{problem:avg-stretch} does not seem to have a non-trivial solution.

To rephrase \cref{problem:avg-stretch}, we are interested in determining a tight upper bound on the $\avgstretch$ that holds uniformly for all sets $A \seq \HC_n$ of density $1/2$.
Note that since the diameter of $\HC_n$ is $n$, for any set $A \seq \HC_n$ of density $1/2$ and any bijection $\phi \colon \HC_{n-1} \to A$ it holds that $\avgstretch(\phi) \leq n$. It is natural to ask how tight this bound is, i.e.,
whether there exists $A \seq \HC_n$ of density $1/2$ such that any bijection $\phi \colon \HC_{n-1} \to A$
requires linear average stretch.

It is consistent with our current knowledge (though hard to believe) that for any set $A$ of density $1/2$ there is a mapping $\phi \colon \HC_{n-1} \to A$ with $\avgstretch(\phi) \leq 2$.
The strongest lower bound we are aware of is for the set $A_{\oplus} = \{x \in \HC_n : \sum_{i}x_i \equiv 0 \pmod 2\}$.
Note that the distance between any two points in $A_{\oplus}$ is at least 2, and hence $\avgstretch(\phi) \geq 2$ for any mapping $\phi \colon \HC_{n-1} \to A_{\oplus}$.
Proving a lower bound strictly greater than 2 for any set $A$ is an open problem,
and prior to this work we are not aware of any sublinear upper bounds that apply uniformly to all sets.

Most of the research on metric embedding, we are aware of, focuses on worst case stretch.
For a survey on metric embeddings of finite spaces see~\cite{Li02}.
There has been a lot of research on the question of embedding into the Boolean cube.
For example, see~\cite{AB07,HLN87} for work on embeddings between random subsets of the Boolean cube,
and~\cite{G88} for isometric embeddings of arbitrary graphs into the Boolean cube.

\subsection{A uniform upper bound on the average stretch}\label{sec:our-results-sqrt-bound}
We prove a non-trivial uniform upper bound on the average stretch of
a mapping $\phi \colon \HC_{n-1} \to A$ that applies to all sets $A \seq \HC_n$ of density $1/2$.

\begin{maintheorem}\label{thm:sqrt-n-avg-stretch}
    For any set $A \seq \HC_n$ of density $\mu_n(A) = 1/2$,
    there exists a bijection $\phi \colon \HC_{n-1} \to A$ such that $\avgstretch(\phi) = O(\sqrt{n})$.
\end{maintheorem}

Toward this goal we prove a stronger result bounding the average transportation distance between two arbitrary sets of density $1/2$.
Specifically, we prove the following theorem.
\begin{maintheorem}\label{thm:sqrtn-transportation-dist}
    For any two sets $A,B \seq \HC_n$ of density $\mu_n(A) = \mu_n(B) = 1/2$,
    there exists a bijection $\phi \colon A \to B$ such that $\E[\dist(x,\phi(x))] \leq \sqrt{2n}$.
\end{maintheorem}

Note that \cref{thm:sqrt-n-avg-stretch} follows immediately from \cref{thm:sqrtn-transportation-dist} by the following simple argument.
\begin{proposition}\label{prop:avg-dist-avg-stretch}
Fix a bijection $\phi \colon \HC_{n-1} \to A$.
Then $\avgstretch(\phi) \leq 2\E_{x \in \HC_{n-1}}[\dist(x,\phi(x))] + 1$.
\end{proposition}
\begin{proof} Using the triangle inequality we have
    \begin{eqnarray*}
    \avgstretch(\phi) & = & \E_{\substack{x \in \HC_{n-1} \\ i \in  [n-1]}}[\dist(\phi(x),\phi(x+e_i))] \\
      & \leq & \E[\dist(x, \phi(x)) + \dist(x, x+e_i) + \dist(x+e_i,\phi(x+e_i))] \\
      & = & \E[\dist(x,\phi(x))] + 1 + \E[\dist(x+e_i,\phi(x+e_i))] \\
      & = & 2\E[\dist(x,\phi(x))] + 1,
    \end{eqnarray*}
    as required.
\end{proof}

\subsection{Bounds on the average stretch for specific sets}\label{sec:our-results-specific-sets}

Next, we study two specific subsets of $\HC_n$ defined by Boolean functions commonly studied in the field ``Analysis of Boolean functions'' \cite{OD14}.
Specifically, we study two monotone noise-sensitive functions: the \emph{recursive majority of 3's}, and the \emph{tribes} function.

It was suggested in \cite{BCS} that the set of ones of these functions $A_f = f^{-1}(1)$ may be
such that any mapping $\phi \colon \HC_{n-1} \to A_f$ requires large $\avgstretch$.
We show that for the recursive majority function there is a mapping $\phi_{\recmaj[]} \colon \HC_{n-1} \to \recmaj[]^{-1}(1)$
with $\avgstretch(\phi_{\recmaj[]}) = O(1)$.
For the tribes function we show a mapping $\phi_{\tribes} \colon \HC_{n-1} \to \tribes^{-1}(1)$
with $\avgstretch(\phi_{\tribes}) = O(\log(n))$.
Below we formally define the functions, and discuss our results.

\subsubsection{Recursive majority of 3's}\label{sec:recmaj-def}
The recursive majority of 3's function is defined as follows.
\begin{definition}
    Let $k \in \N$ be a positive integer.
    Define the function \defemph{recursive majority of 3's} $\recmaj \colon \HC_{3^k} \to \Bits$ as follows.
    \begin{itemize}
      \item
    For $k = 1$ the function $\recmaj[1]$ is the majority function on the $3$ input bits.
      \item
    For $k > 1$ the function $\recmaj \colon \HC_{3^k} \to \Bits$ is defined recursively as follows.
    For each $x \in \HC_{3^k}$ write $x = x^{(1)} \circ x^{(2)} \circ x^{(3)}$, where each $x^{(\third)} \in \HC_{3^{k-1}}$ for each $\third \in [3]$.
    Then, $\recmaj(x) = \maj(\recmaj[k-1](x^{(1)}), \recmaj[k-1](x^{(2)}), \recmaj[k-1](x^{(3)}))$.

    \end{itemize}
\end{definition}

Note that $\recmaj(x) = 1-\recmaj({\bf 1}-x)$ for all $x \in \HC_n$,
and hence the density of the set $\Arecmaj = \{x \in \HC_{n} : \recmaj(x) = 1\}$ is
$\mu_n(\Arecmaj) = 1/2$.
We prove the following result regarding the set $\Arecmaj$.

\begin{maintheorem}\label{thm:recmaj-avgstretch}
    Let $k$ be a positive integer, let $n = 3^k$, and let $\Arecmaj = \{x \in \HC_{n} : \recmaj(x) = 1\}$.
    There exists a mapping $\phi_{\recmaj} \colon \HC_{n-1} \to \Arecmaj$ such that $\avgstretch(\phi_{\recmaj}) \leq 20$.
\end{maintheorem}

\subsubsection{The tribes function}\label{sec:tribes-def}
The tribes function is defined as follows.
\begin{definition}
    Let $s, w \in \N$ be two positive integers, and let $n = s \cdot w$.
    The function $\tribes \colon \HC_n \to \Bits$ is defined as a DNF consisting of $s$ disjoint clauses of width $w$.
    \begin{equation*}
        \tribes(x_1,x_2,\dots,x_w; \dots ; x_{(s-1)w+1}\dots x_{sw})
        = \bigvee_{i=1}^s
        (x_{(i-1)w+1} \wedge x_{(i-1)w+2} \wedge \dots \wedge x_{iw}).
    \end{equation*}
    That is, the function $\tribes$ partitions $n = sw$ inputs into $s$ disjoint ``tribes'' each of size $w$,
    and returns 1 if and only if at least one of the tribes ``votes'' 1 unanimously.
\end{definition}

It is clear that $\Pr_{x \in \HC_n}[\tribes(x) = 1] = 1 - (1-2^{-w})^s$.
The interesting settings of parameters $w$ and $s$ are such that the function is close to balanced, i.e., this probability is close to $1/2$.
Given $w \in \N$, let $s = s_w = \ln(2) 2^w \pm \Theta(1)$ be the largest integer such that $1 - (1-2^{-w})^s \leq 1/2$.
For such choice of the parameters we have $\Pr_{x \in \HC_n}[\tribes(x) = 1] = \half - O\left(\frac{\log(n)}{n}\right)$ (see, e.g., \cite[Section 4.2]{OD14}).

Consider the set $\Atribes = \{x \in \HC_n : \tribes(x) = 1\}$.
Since the density of $\Atribes$ is not necessarily equal to $1/2$, we cannot talk about a bijection from $\HC_{n-1}$ to $\Atribes$.
In order to overcome this technical issue, let $\Astribes$ be an arbitrary superset of $\Atribes$ of density $1/2$.
We prove that there is a mapping $\phi_\tribes$ from $\HC_{n-1}$ to $\Astribes$ with average stretch $\avgstretch(\phi_\tribes) = O(\log(n))$.
In fact, we prove a stronger result, namely that the average transportation distance of $\phi_\tribes$ is $O(\log(n))$.

\begin{maintheorem}\label{thm:tribes-avgstretch}
    Let $w$ be a positive integer, and let $s$ be the largest integer such that $1 - (1-2^{-w})^s \leq 1/2$.
    For $n=s \cdot w$ let $\tribes \colon \HC_n \to \Bits$ be defined as a DNF consisting of $s$ disjoint clauses of width $w$.
    Let $\Atribes = \{x \in \HC_n : \tribes(x) = 1\}$, and
    let $\Astribes \seq \HC_n$ be an arbitrary superset of $\Atribes$ of density $\mu_n(\Astribes) = 1/2$.
    Then, there exists a bijection $\phi_{\tribes} \colon \HC_{n-1} \to \Astribes$ such that
    $\E[\dist(x,\phi_{\tribes}(x))] = O(\log(n))$.
    In particular, $\avgstretch(\phi_\tribes) = O(\log(n))$.
\end{maintheorem}

\subsection{Roadmap}
The rest of the paper is organized as follows.
We prove \cref{thm:sqrtn-transportation-dist} in \cref{sec:sqrtn-transportation}.
In \cref{sec:recmaj} we prove \cref{thm:recmaj-avgstretch}, and in \cref{sec:tribes} we prove \cref{thm:tribes-avgstretch}.

\section{Proof of \cref{thm:sqrtn-transportation-dist}}
\label{sec:sqrtn-transportation}

We provide two different proofs of \cref{thm:sqrtn-transportation-dist}.
The first proof, in \cref{sec:gale-shapley} shows a slightly weaker bound of $O\left(\sqrt{n \ln(n)}\right)$
on the average stretch using the Gale-Shapley result on the \emph{stable marriage problem}.
The idea of using the stable marriage problem was suggested in \cite{BCS}, and we implement this approach.
Then, in \cref{sec:wasserstein}, we show the bound of $O(\sqrt{n})$ by relating
the average stretch of a mapping between two sets to known estimates on the Wasserstein distance on the hypercube.

\subsection{Upper bound on the average transportation distance using stable marriage}\label{sec:gale-shapley}

Recall the Gale-Shapley theorem on the \emph{stable marriage problem}.
In the stable marriage problem we are given two sets of elements $A$ and $B$ each of size $N$.
For each element $a \in A$ (resp.\ $b \in B$) we have a ranking of the elements of $B$ (resp.\ $A$)
given as an bijection $\rk_a : A \to [N]$ ($\rk_b : B \to [N]$) representing the preferences of each $a$ (resp.\ $b$).
A matching (or a bijection) $\phi \colon A \to B$ is said to be \emph{unstable} if there are $a,a' \in A$, and $b,b' \in B$
such that $\phi(a) = b'$, $\phi(a') = b$, but $\rk_a(b) < \rk_a(b')$, and $\rk_b(a) < \rk_b(a')$;
that is, both $a$ and $b$ would prefer to be mapped to each other over their mappings given by $\phi$.
We say that a matching $\phi : A \to B$ is \emph{stable} otherwise.

\begin{theorem}[Gale-Shapley theorem]\label{thm:gale-shapley}
    For any two sets $A$ and $B$ of equal size and any choice of rankings
    for each $a \in A$ and $b \in B$ there exists a stable matching $m \colon A \to B$.
\end{theorem}

For the proof of \cref{thm:sqrtn-transportation-dist} the sets $A$ and $B$ are subsets of $\HC_n$ of density 1/2.
We define the preferences between points based on the distances between them in $\HC_n$.
That is, for each $a \in A$ we have $\rk_a(b) < \rk_a(b')$ if and only if $\dist(a,b) < \dist(a,b')$ with ties broken arbitrarily.
Similarly, for each $b \in B$ we have $\rk_b(a) < \rk_b(a')$ if and only if $\dist(a,b) < \dist(a',b)$ with ties broken arbitrarily.

Let $\phi \colon A \to B$ be a bijection.
We show that if $\E_{x \in A}[\dist(x, \phi(x))]> 3k$ for $k=\ceil{\sqrt{n \ln(n)}}$, then $\phi$ is not a stable matching.
Consider the set
\begin{equation*}
    F := \{x \in A \mid  \dist(x, \phi(x)) \geq k\}.
\end{equation*}
Note that since the diameter of $\HC_n$ is $n$,
and $\E_{x \in A}[\dist(x, \phi(x))] > 3k$,
it follows that $\mu_n(F) > \frac{k}{n}$.
Indeed, we have $3k < \E_{x \in A}[\dist(x, \phi(x))] \leq n \cdot \frac{\mu_n(F)}{\mu_n(A)} + k \cdot (1-\frac{\mu_n(F)}{\mu_n(A)}) \leq n \cdot \frac{\mu_n(F)}{\mu_n(A)} + k$,
and thus $\mu_n(F) > \frac{2k}{n} \cdot \mu_n(A)$.
Next, we use Talagrand's concentration inequality.
\begin{theorem}[{\cite[Proposition 2.1.1]{Talagrand95}}]\label{thm:talagrand}
    Let $k \leq n$ be two positive integers, and let $F \seq \HC_n$.
    Let $F_{\geq k} = \{x \in \HC_n : \dist(x,y) \geq k \ \forall y \in F \}$ be the set of all $x \in \HC_n$
    whose distance from $F$ is at least $k$.
    Then $\mu_n(F_{\geq k}) \leq e^{-k^2/n} / \mu_n(F)$.
\end{theorem}
By \cref{thm:talagrand} we have $\mu_n(F_{\geq k}) \leq e^{-k^2/n} / \mu_n(F)$,
and hence, for $k = \ceil{\sqrt{n \ln(n)}}$ it holds that
\begin{equation*}
    \mu_n(F_{\geq k}) \leq e^{-\ln(n)} / \mu_n(F) \leq (1/n)/(k/n) = 1/k.
\end{equation*}
In particular, since $\mu_n(\phi(F)) =\mu_n(F) > k/n > 1/k \geq \mu_n(F_{\geq k})$,
there is some $b \in \phi(F)$ that does not belong to $F_{\geq k}$.
That is, there is some $a \in F$ and $b \in \phi(F)$ such that $\dist(a,b) < k$.
On the other hand, for $a' = \phi^{-1}(b)$,
by the definition of $F$ we have $\dist(a,\phi(a)) \geq k$ and $\dist(a',b=\phi(a')) \geq k$,
and hence $\phi$ is not stable, as $a$ and $b$ prefer to be mapped to each other over their current matching.
Therefore, in a stable matching $\E_{x \in A}[\dist(x, \phi(x))] \leq 3\ceil{\sqrt{n \ln(n)}}$,
and by the Gale-Shapley theorem such a matching does, indeed, exist.\qed

\subsection{Proof of \cref{thm:sqrtn-transportation-dist} using transportation theory}\label{sec:wasserstein}

Next we prove \cref{thm:sqrtn-transportation-dist}, by relating our problem to a known estimate on the Wasserstein distance between two measures on the hypercube.
Recall that the $\ell_1$-Wasserstein distance between two measures $\mu$ and $\nu$ on $\HC_n$ is defined as
\begin{equation*}
    W_1(\mu, \nu) = \inf_q \sum_{x, y} \dist(x,y) q(x,y),
\end{equation*}
where the infimum is taken over all couplings $q$ of $\mu$ and $\nu$, i.e., $\sum_{y'} q(x,y')  = \mu(x) $ and $\sum_{x'} q(x',y) = \nu(y)$ for all $x,y \in \HC_n$.
That is, we consider an optimal coupling $q$ of $\mu$ and $\nu$
minimizing $\E_{(x,y) \sim q}[\dist(x,y)]$, the expected distance between $x$ and $y$,
where $x$ is distributed according to $\mu$ and $y$ is distributed according to $\nu$.

We prove the theorem using the following two claims.
\begin{claim}\label{claim:W_1-bijection}
    Let $\mu_A$ and $\mu_B$ be uniform measures over the sets $A$ and $B$ respectively.
    Then, there exists a bijection $\phi$ from $A$ to $B$ such that $\E[\dist(x,\phi(x))]  = W_1(\mu_A,\mu_B)$.
\end{claim}
\begin{claim}\label{claim:W_1-less-sqrtn}
    Let $\mu_A$ and $\mu_B$ be uniform measures over the sets $A$ and $B$ respectively.
    Then $W_1(\mu_A, \mu_B) \leq \sqrt{2n}$.
\end{claim}

\begin{proof}[Proof of \cref{claim:W_1-bijection}]
Observe that any bijection $\phi$ from $A$ to $B$ naturally defines a coupling $q$ of $\mu_A$ and $\mu_B$,
defined as
    \begin{equation*}
        q(x, y) = \begin{cases}
                        \frac{1}{|A|} & \mbox{if $x \in A$ and $y=\phi(x)$},  \\
                         0 & \mbox{otherwise.}
                       \end{cases}
    \end{equation*}
Therefore, $W_1(\mu_A,\mu_B) \leq \E_{x \in A}[\dist(x,\phi(x))]$.

For the other direction note that in the definition of $W_1$ we are looking for the infimum of the linear function
$L(q) = \sum_{(x,y) \in A \times B} \dist(x,y) q(x,y)$, where the infimum is taken over the \emph{Birkhoff polytope} of all $n \times n$ doubly stochastic matrices.
By the Birkhoff-von Neumann theorem~\cite{Bir46,vonNeumann53,Konig36} this polytope is the convex hull whose extremal points are precisely the permutation matrices.
The optimum is obtained on such an extremal point, and hence there exists a bijection $\phi$ from $A$ to $B$
such that $W_1(\mu_A,\mu_B) = \E[\dist(x,\phi(x))]$.
\end{proof}

\begin{proof}[Proof of \cref{claim:W_1-less-sqrtn}]
    The proof of the claim follows rather directly from the techniques in \emph{transportation theory} (see~\cite[Section 3.4]{RaSa15}).
    Specifically, using Definition 3.4.2 and combining Proposition 3.4.1, Equation 3.4.42, and Proposition 3.4.3,
    where $\mathcal{X} = \Bits$, and $\mu$ is the uniform distribution on $\mathcal{X}$ we have the following theorem.
    \begin{theorem}\label{thm:transport}
        Let $\nu$ be an arbitrary distribution on the discrete hypercube $\HC_n$,
        and let $\mu_n$ be the uniform distribution on $\HC_n$. Then
        \begin{equation*}
            W_1(\nu,\mu_n) \leq \sqrt{\half n \cdot D(\nu \mid \mid \mu_n)},
        \end{equation*}
        where $D(\nu \mid \mid \mu)$ is the Kullback-Leibler divergence defined as $D(\nu \mid \mid \mu) = \sum_{x} \nu(x) \log(\frac{\nu(x)}{\mu(x)})$.
    \end{theorem}
    In particular, by letting $\nu = \mu_A$ be the uniform distribution over the set $A$ of cardinality $2^{n-1}$, we have
    $D(\mu_A \mid \mid \mu_n) = \sum_{x \in A} \mu_A(x)\log\left(\frac{\mu_A(x)}{\mu_n(x)}\right) = \sum_{x \in A} \frac{1}{|A|}\log(2) = 1$,
    and hence $W_1(\mu_n,\mu_A) \leq \sqrt{\half n \cdot D(\mu_n \mid \mid \nu)} = \sqrt{n/2}$.
    Analogously, we have $W_1(\mu_n,\mu_B) \leq \sqrt{n/2}$.
    Therefore, by the triangle inequality, we conclude that $W_1(\mu_A,\mu_B) \leq W_1(\mu_A,\mu_n) + W_1(\mu_n,\mu_B) \leq \sqrt{2n}$, as required.
\end{proof}

This completes the proof of \cref{thm:sqrtn-transportation-dist}.

\section{Average stretch for recursive majority of 3's}
\label{sec:recmaj}

In this section we prove \cref{thm:recmaj-avgstretch}, showing a mapping from $\HC_n$ to $\Arecmaj$ with constant average stretch.
The key step in the proof is the following lemma.
\begin{lemma}\label{lemma:f-k-mapping}
Let $k$ be a positive integer, and let $n = 3^k$.
There exists $f_k \colon \HC_n \to \Arecmaj$ satisfying the following properties.
    \begin{enumerate}
      \item $f_k(x) =x$ for all $x \in \Arecmaj$.
      \item For each $x \in \Arecmaj$ there is a unique $z \in \Zrecmaj \DefineEqual \HC_n \setminus \Arecmaj$ such that $f_k(z) = x$.
      \item For every $i \in [n]$ we have $\E_{x \in \HC_n}[\dist(f_k(x),f_k(x+e_i)) ] \leq 10$. \label{item:f-k-stretc}
    \end{enumerate}
\end{lemma}

We postpone the proof of \cref{lemma:f-k-mapping} for now, and show how it implies \cref{thm:recmaj-avgstretch}.

\begin{proof}[Proof of \cref{thm:recmaj-avgstretch}]
  Let $f_k$ be the mapping from \cref{lemma:f-k-mapping}.
  Define $\psi_0,\psi_1 \colon \HC_{n-1} \to \Arecmaj$ as $\psi_b(x) = f_k(x \circ b)$,
  where $x \circ b \in \HC_n$ is the string obtained from  $x$ by appending to it $b$ as the $n$'th coordinate.

  The mappings $\psi_0,\psi_1$ naturally induce a bipartite graph $G = (V,E)$, where $V = \HC_{n-1} \cup \Arecmaj$
  and $E = \{(x,\psi_b(x)) : x \in \HC_{n-1}, b \in \Bits \}$, possibly, containing parallel edges.
  Note that by the first two items of \cref{lemma:f-k-mapping} the graph $G$ is 2-regular.
  Indeed, for each $x \in \HC_n$ the neighbours of $x$ are $N(x) = \{ \psi_0(x) = f_k(x \circ 0), \psi_1(x) = f_k(x \circ 1) \}$,
  and for each $y \in \Arecmaj$ there is a unique $x \in \Arecmaj$ and a unique $z \in \Zrecmaj$ such that $f_k(x) = f_k(z) = 1$, and
  hence $N(y) = \{ x_{[1,\dots,n-1]}, z_{[1,\dots,n-1]} \}$.

  Since the bipartite graph $G$ is 2-regular, it has a perfect matching.
  Let $\phi$ be the bijection from $\HC_{n-1}$ to $\Arecmaj$ induced by a perfect matching in $G$,
  and for each $x \in \HC_n$ let $b_x \in \HC_n$ be such that $\phi(x) = \psi_{b_x}(x)$.
  We claim that $\avgstretch(\phi) = O(1)$. Let $x \sim x'$ be uniformly random vertices in $\HC_{n-1}$ that differ in exactly one coordinate,
  and let $r \in \Bits$ be uniformly random. Then
  \begin{eqnarray*}
  \E[\dist(\phi(x), \phi(x'))]
      & = & \E[\dist(f_k(x \circ b_x), f_k(x' \circ b_{x'}))] \\
      & \leq & \E[\dist(f_k(x \circ b_x), f_k(x \circ r))] + \E[\dist(f_k(x \circ r), f_k(x' \circ r))]  \\
      & & + \E[\dist(f_k(x' \circ r), f_k(x' \circ b_{x'}))].
  \end{eqnarray*}
  For the first term, since $r$ is equal to $b_x$ with probability $1/2$ by \cref{lemma:f-k-mapping} \cref{item:f-k-stretc}
  we get that $\E[\dist(f_k(x \circ b_x), f_k(x \circ r))]\leq 5$.
  Analogously the third term is bounded by $5$.
  In the second term we consider the expected distance between $f(\cdot)$ applied on inputs that differ in a random coordinate $i \in [n-1]$,
  which is at most $10$, again, by \cref{lemma:f-k-mapping} \cref{item:f-k-stretc}.
  Therefore $\E[\dist(\phi(x), \phi(x'))] \leq 20$.
\end{proof}

We return to the proof of \cref{lemma:f-k-mapping}.
\begin{proof}[Proof of \cref{lemma:f-k-mapping}]
    Define $f_k \colon \HC_n \to \Arecmaj$ by induction on $k$.
    For $k = 1$ define $f_1$ as
    \begin{align*}
        000 &\mapsto 110 \\
        100 &\mapsto 101 \\
        010 &\mapsto 011 \\
        001 &\mapsto 111 \\
        x  &\mapsto x \quad \textrm{for all } x \in \{110,101,011,111\}.
    \end{align*}
    That is, $f_1$ acts as the identity map for all $x \in \Arecmaj[1]$, and maps all inputs in $\Zrecmaj[1]$ to $\Arecmaj[1]$ in a one-to-one way.
    Note that $f_1$ is a non-decreasing mapping, i.e., $(f_1(x))_i \geq x_i$ for all $x \in \HC_3$ and $i \in [3]$.

    For $k>1$ define $f_k$ recursively using $f_{k-1}$ as follows.
    For each $\third \in [3]$, let $T_\third = [(\third-1) \cdot 3^{k-1}+1, \ldots,\third \cdot 3^{k-1}]$ be the $\third$'th third of the interval $[3^k]$.
    For $x \in \HC_{3^k}$, write $x = x^{(1)} \circ x^{(2)} \circ x^{(3)}$,
    where $x^{(\third)} = x_{T_\third} \in \HC_{3^{k-1}}$ is the $\third$'th third of $x$.
    Let $y = (y_1,y_2,y_3) \in \Bits^3$ be defined as $y_\third = \recmaj[k-1](x^{(\third)})$, and let $w =(w_1,w_2,w_3) = f_1(y) \in \Bits^3$.
    Define
    \begin{equation*}
        f_{k-1}^{(r)}(x^{(r)}) = \begin{cases}
                         f_{k-1}(x^{(\third)}) & \mbox{if $w_\third \neq y_\third$},  \\
                         x^{(\third)} & \mbox{otherwise}.
                       \end{cases}
    \end{equation*}
    Finally, the mapping $f_k$ is defined as
    \begin{equation*}
        f_k(x) =
          f_{k-1}^{(1)}(x^{(1)}) \circ f_{k-1}^{(2)}(x^{(2)}) \circ f_{k-1}^{(3)}(x^{(3)}).
    \end{equation*}
    That is, if $\recmaj(x) = 1$ then $w = y$, and hence $f_k(x) = x$,
    and otherwise, $f_{k-1}^{(r)}(x^{(r)}) \neq x^{(r)}$ for all $r \in [3]$ where $y_r = 0$ and $w_r = 1$.

    Next we prove that $f_k$ satisfies the properties stated in \cref{lemma:f-k-mapping}.
    \begin{enumerate}
      \item It is clear from the definition that if $\recmaj(x) = 1$, then $w = y$, and hence $f_k(x) = x$.

      \item Next, we prove by induction on $k$ that the restriction of $f_k$ to $\Zrecmaj$ induces a bijection.
      For $k = 1$ the statement clearly holds.
      For $k > 2$ suppose that the restriction of $f_{k-1}$ to $\Zrecmaj[k-1]$ induces a bijection.
      We show that for every $x \in \Arecmaj$ the mapping $f_k$ has a preimage of $x$ in $\Zrecmaj$.
      Write $x = x^{(1)} \circ x^{(2)} \circ x^{(3)}$,
      where $x^{(\third)} = x_{T_\third} \in \HC_{3^{k-1}}$ is the $\third$'th third of $x$.
      Let $w = (w_1,w_2,w_3)$ be defined as $w_\third = \recmaj[k-1](x^{(\third)})$.
      Since $x \in \Arecmaj$ it follows that $w \in \{110,101,011,111\}$.
      Let $y =(y_1,y_2,y_3) \in \Zrecmaj[1]$ such that $f_1(y) = w$.

      For each $\third \in [3]$ such that $w_\third = 1$ and $y_\third = 0$
      it must be the case that $x^{(\third)} \in \Arecmaj[k-1]$,
      and hence, by the induction hypothesis, there is some $z^{(\third)} \in \Zrecmaj[k-1]$ such that $f_{k-1}(z^{(\third)}) = x^{(\third)}$.
      For each $\third \in [3]$ such that $y_\third = w_\third$ define $z^{(\third)} = x^{(\third)}$.
      Since $y =(y_1,y_2,y_3) \in \Zrecmaj[1]$, it follows that $z = z^{(1)} \circ z^{(2)} \circ z^{(3)} \in \Zrecmaj$.
      It is immediate by the construction that, indeed, $f_k(z) = x$.

      \item Fix $i \in [3^k]$. In order to prove $\E[\dist(f_k(x),f_k(x+e_i)) ] = O(1)$
      consider the following events.
      \begin{eqnarray*}
        & E_1 = \{ \recmaj(x) = 1 = \recmaj(x+e_i)\}, \\
        & E_2 = \{ \recmaj(x) = 0, \recmaj(x+e_i) = 1\}, \\
        & E_3 = \{ \recmaj(x) = 1, \recmaj(x+e_i) = 0\}, \\
        & E_4 = \{ \recmaj(x) = 0 = \recmaj(x+e_i)\}.
      \end{eqnarray*}
      Then $\E[\dist(f_k(x), f_k(x+e_i))] = \sum_{j=1,2,3,4}\E[\dist(f_k(x), f_k(x+e_i)) | E_j] \cdot \Pr[E_j]$.
      The following three claims prove an upper bound on $\E[\dist(f_k(x), f_k(x+e_i))]$.

      \begin{claim}\label{claim:recmaj-E1}
        $\E[\dist(f_k(x), f_k(x+e_i)) | E_1] = 1$.
      \end{claim}

      \begin{claim}\label{claim:recmaj-E2}
        $\E[\dist(f_k(x), f_k(x+e_i)) | E_2] \leq 2 \cdot 1.5^k$.
      \end{claim}

      \begin{claim}\label{claim:recmaj-E4}
        $\E[\dist(f_k(x), f_k(x+e_i)) | E_4] \cdot \Pr[E_4] \leq 8$.
      \end{claim}
  \end{enumerate}

  By symmetry, it is clear that
  $\E[\dist(f_k(x), f_k(x+e_i)) | E_2] = \E[\dist(f_k(x), f_k(x+e_i)) | E_3]$.
  Note also that $\Pr[E_1] < 0.5$ and $\Pr[E_2 \cup E_3] = 2^{-k}$.%
  \footnote{Indeed, note that
  $\Pr[E_2 \cup E_3] = \Pr[\recmaj(x) \neq \recmaj(x+e_i)]$, and suppose for concreteness that $i=1$.
  We claim that $\Pr[\recmaj(x) \neq \recmaj(x+e_1)]=2^{-k}$, which can be seen by induction on $k$
  using the recurrence $\Pr[\recmaj(x) \neq \recmaj(x+e_1)] = \Pr[\recmaj[k-1](x^{(2)}) \neq \recmaj[k-1](x^{(3)})]
  \cdot \Pr[\recmaj[k-1](x^{(1)}) \neq \recmaj[k-1](x^{(1)}+e_1)]$ $= (1/2) \cdot \Pr[\recmaj[k-1](x^{(1)}) \neq \recmaj[k-1](x^{(1)}+e_1)] = (1/2)\cdot 2^{-(k-1)} = 2^{-k}$.
  }
  Therefore, the claims above imply that
  \begin{equation*}
        \E[\dist(f_k(x), f_k(x+e_i))] = \sum_{j=1,2,3,4}\E[\dist(f_k(x), f_k(x+e_i)) | E_i] \cdot \Pr[E_i]
        \leq 1 \cdot 0.5 + 2 \cdot 1.5^k \cdot 2^{-k} + 8 \leq 10,
  \end{equation*}
  which completes the proof of \cref{lemma:f-k-mapping}.
\end{proof}

  Next we prove the above claims.

  \begin{proof}[Proof of \cref{claim:recmaj-E1}]
    If $E_1$ holds then $\dist(f_k(x), f_k(x+e_i)) = \dist(x,x+e_i) = 1$.
  \end{proof}

  \begin{proof}[Proof of \cref{claim:recmaj-E2}]
    We prove first that
    \begin{equation}\label{eq:dist-1.5^k}
        \E[\dist(x,f_k(x)) | \recmaj(x) = 0] = 1.5^k.
    \end{equation}
    The proof is by induction on $k$.
    For $k = 1$ we have $\E[\dist(x,f_1(x)) | \recmaj(x) = 0] = 1.5$
    as there are two inputs $x \in \Zrecmaj$ with $\dist(x,f_1(x)) = 1$
    and two $x$'s in $\Zrecmaj$ with $\dist(x,f_1(x)) = 2$.
    For $k > 1$ suppose that $\E[\dist(x,f_{k-1}(x)) | \recmaj[k-1](x)] = 1.5^{k-1}$.
    Write each $x \in \HC_{3^k}$ as $x = x^{(1)} \circ x^{(2)} \circ x^{(3)}$,
    where $x^{(\third)} = x_{T_\third} \in \HC_{3^{k-1}}$ is the $\third$'th third of $x$,
    and let $y = (y_1,y_2,y_3)$ be defined as $y_\third = \recmaj[k-1](x^{(\third)})$.
    Since $\E_{x \in H_{3^{k-1}}}[\recmaj[k-1](x)] = 0.5$, it follows that
    for a random $z \in \Zrecmaj$ each $y \in \{000, 100, 010, 001\}$
    happens with the same probability $1/4$, and hence, using the induction hypothesis we get
    \begin{eqnarray*}
      \E[\dist(x,f_{k}(x)) | \recmaj(x) = 0]
      &=& \Pr[y \in \{100, 010\} | \recmaj(x) = 0] \times 1.5^{k-1} \\
      && + \Pr[y \in \{000, 001\} | \recmaj(x) = 0] \times 2 \cdot 1.5^{k-1} \\
      &=& 1.5^k,
    \end{eqnarray*}
    which proves \cref{eq:dist-1.5^k}.

    Next we prove that%
    \footnote{Note that \cref{eq:dist-cond-on-boundary-1.5^k} can be thought of \cref{eq:dist-1.5^k} conditioned on the event $\recmaj(x+e_i) = 1$,
    which happens with probability only $2^{-k}$.
    A naive application of Markov's inequality would only say that
    $\E[\dist(x, f_k(x)) | E_2] \leq 1.5^k \cdot 2^k$, which would not suffice for us.
    \cref{eq:dist-cond-on-boundary-1.5^k} says that the expected distance is comparable to $1.5^k$ even when conditioning on this small event.}
    \begin{equation}\label{eq:dist-cond-on-boundary-1.5^k}
        \E[\dist(x, f_k(x)) | E_2] \leq \sum_{j=0}^{k-1} 1.5^j = 2 \cdot (1.5^k - 1).
    \end{equation}
    Note that \cref{eq:dist-cond-on-boundary-1.5^k} proves \cref{claim:recmaj-E2}.
    Indeed, if $E_2$ holds then using the triangle inequality we have
    $\dist(f_k(x), f_k(x+e_i)) \leq \dist(f_k(x), x) + \dist(x,x+e_i) + \dist(x+e_i, f_k(x+e_i))
    = \dist(f_k(x), x) + 1$,
    and hence
    \begin{equation*}
      \E[\dist(f_k(x), x)| E_2] + 1
      \leq 2 \cdot (1.5^k - 1) + 1 < 2 \cdot 1.5^k,
    \end{equation*}
    as required.

    We prove \cref{eq:dist-cond-on-boundary-1.5^k} by induction on $k$.
    For $k = 1$ \cref{eq:dist-cond-on-boundary-1.5^k} clearly holds.
    For the induction step let $k > 1$.
    As in the definition of $f_k$ write each $x \in \HC_{3^k}$ as $x = x^{(1)} \circ x^{(2)} \circ x^{(3)}$,
    where $x^{(\third)} = x_{T_\third}$ is the $\third$'th third of $x$,
    and let $y = (y_1,y_2,y_3)$ be defined as $y_\third = \recmaj[k-1](x^{(\third)})$.

    Let us suppose for concreteness that $i \in T_1$.
    (The cases of $i \in T_2$ and $i \in T_3$ are handled similarly.)
    Note that if $\recmaj(x) = 0$, $\recmaj(x+e_i) = 1$, and $i \in T_1$, then $y \in \{010,001\}$.
    We consider each case separately.
    \begin{enumerate}
    \item
    Suppose that $y = 010$. Then $w = f(y) =011$, and hence $f(x)$ differs from $x$ only in $T_3$.
    Taking the expectation over $x$ such that $\recmaj(x) = 0$ and $\recmaj(x+e_i) = 1$
    by \cref{eq:dist-1.5^k} we get $\E[\dist(x,f(x)) | E_2, y = 010] = \E[\dist(f_{k-1}(x^{(3)}), x^{(3)}) | \recmaj[k-1](x^{(3)})=0] = 1.5^{k-1}$.
    \item
    If $y = 001$, then $w = f_1(y) = 111$, and $f(x)$ differs from $x$ only in $T_1 \cup T_2$.
    Then
    \begin{eqnarray*}
        \E[\dist(x,f(x)) | E_2 , y = 001]
        &=& \E[\dist(f_{k-1}(x^{(1)}), x^{(1)}) | E_2, y = 001] \\
        && + \E[\dist(f_{k-1}(x^{(2)}), x^{(2)}) | E_2, y = 001].
    \end{eqnarray*}
    Denoting by $E'_2$ the event that $\recmaj[k-1](x^{(1)}) = 0, \recmaj[k-1](x^{(1)}+e_i) = 1$ (i.e., the analogue of the event $E_2$ applied on $\recmaj[k-1]$),
    we note that
    \begin{equation*}
        \E[\dist(f_{k-1}(x^{(1)}), x^{(1)}) | E_2, y = 001] = \E[ \dist(f_{k-1}(x^{(1)}), x^{(1)}) | E'_2],
    \end{equation*}
    which is upper bounded by $\sum_{j=0}^{k-2} 1.5^j$ using the induction hypothesis.
    For the second term we have
    \begin{equation*}
        \E[\dist(f_{k-1}(x^{(2)}), x^{(2)}) | E_2, y = 001] = \E[ \dist(f_{k-1}(x^{(2)}), x^{(2)}) | \recmaj[k-1](x^{(2)}) = 0],
    \end{equation*}
    which is at most $1.5^{k-1}$ using \cref{eq:dist-1.5^k}.
    Therefore, for $y = 001$ we have
    \begin{equation*}
        \E[\dist(x,f(x)) | E_2 , y = 001] \leq \sum_{j=0}^{k-2} 1.5^j + 1.5^{k-1}.
    \end{equation*}
    \end{enumerate}
    Using the two cases for $y$ we get
    \begin{eqnarray*}
        \E[\dist(x, f_k(x)) | E_2]
        & = & \E[\dist(x, f_k(x)) | E_2, y = 010] \cdot \Pr[y = 010 | E_2] \\
        && + \E[\dist(x, f_k(x)) | E_2, y = 001] \cdot \Pr[y = 001 | E_2] \\
        &\leq & \sum_{j=0}^{k-1} 1.5^j.
    \end{eqnarray*}
    This proves \cref{eq:dist-cond-on-boundary-1.5^k} for the case where $i \in T_1$.
    The other two cases are handled similarly.
    This completes the proof of \cref{claim:recmaj-E2}.
  \end{proof}

  \begin{proof}[Proof of \cref{claim:recmaj-E4}]
    For a coordinate $i \in [n]$ and for $0 \leq j \leq k$ let $\third = \third_i(j) \in \N$ be such that $i \in [(\third-1) \cdot 3^j + 1 , \dots, \third \cdot 3^j]$,
    and denote the corresponding interval by $T_i(j) = [(\third-1) \cdot 3^j + 1 , \dots, \third \cdot 3^j]$.%
    \footnote{For example, for $j = 0$ we have $T_i(0) = \{i\}$
    For $j = 1$ if $i = 1 \pmod 3$ then $T_i(1) = [i,i+1,i+2]$.
    For $j = k-1$ the interval $T_i(k-1)$ is one of the intervals $T_1, T_2, T_3$.
    For $j = k$ we have $T_i(k) = [1,\dots, 3^k]$.}
    These are the coordinates used in the recursive definition of $\recmaj$ by the instance of $\recmaj[j]$ that depends on the $i$'th coordinate.

    For $x \in \HC_n$ and $x' = x + e_i$, define $\mineq(x)$ as
    \begin{equation*}
        \mineq(x) = \begin{cases}
                        \min \{j \in [k] : \recmaj[j](x_{T_i(j)}) = \recmaj[j](x'_{T_i(j)})\} & \mbox{if } \recmaj(x) = \recmaj(x'), \\
                        k+1 & \mbox{if } \recmaj(x) \neq \recmaj(x').
                     \end{cases}
    \end{equation*}
    That is, in the ternary tree defined by the computation of $\recmaj$,
    $\mineq(x)$ is the lowest $j$ on the path from the $i$'th coordinate to the root
    where the computation of $x$ is equal to the computation of $x+e_i$.
    Note that if $x$ is chosen uniformly from $\HC_n$, then
    \begin{equation}\label{eq:mineq-distr}
        \Pr[\mineq = j] = \begin{cases}
                                2^{-j} & \mbox{if } j \in [k], \\
                                2^{-k} & \mbox{if } j = k+1.
                              \end{cases}
    \end{equation}
    Below we show that by conditioning on $E_4$ and on the value of $\mineq$ we get
    \begin{equation}\label{eq:dist-cond-mineq}
        \E[\dist(f_k(x), f_k(x+e_i)) | E_4, \mineq = j] \leq 4 \cdot 1.5^j.
    \end{equation}
    Indeed, suppose that $E_4$ holds. Assume without loss of generality that $x_i = 0$, and let $x' = x + e_i$.
    Note that $f_k(x)$ and $f_k(x')$ differ only on the coordinates in the interval $T_i(\mineq)$.
    Let $w = x_{T_i(\mineq)}$, and define $y = (y_1,y_2,y_3) \in \Bits^3$ as $y_\third = \recmaj[\nu-1](w^{(\third)})$ for each $\third \in [3]$, where $w^{(\third)}$ is the $r$'th third of $w$.
    Similarly, let $w' = x'_{T_i(\mineq)}$, and let $y' = (y'_1,y'_2,y'_3) \in \Bits^3$ be defined as $y'_\third = \recmaj[\nu-1](w'^{(\third)})$ for each $\third \in [3]$.
    This implies that
    \begin{equation*}
        \E[\dist(f_k(x), f_k(x+e_i))| E_4] = \E[\dist(f_\mineq(w)), f_\mineq(w')) | E_4].
    \end{equation*}
    Furthermore, if $\recmaj[\mineq](x_{T_i(\mineq)}) = 1$ (and $\recmaj[\mineq](x'_{T_i(\mineq)}) = 1$),
    then $f_k(x)_{T_i(\mineq)} = x_{T_i(\mineq)}$, and thus $\dist(f_k(x), f_k(x')) = 1$.

    Next we consider the case of $\recmaj(x_{T_i(\mineq)}) = 0$ (and $\recmaj(x'_{T_i(\mineq)}) = 0$).
    Since $x_i = 0$ and $x' = x + e_i$, it must be that $y = 000$ and $y'$ is a unit vector.
    Suppose first that $y' = 100$, i.e., the coordinate $i$ belongs to the first third of $T_i(\mineq)$.
    Write $w = w^{(1)} \circ w^{(2)} \circ w^{(3)}$, where each $w^{(r)}$ is one third of $w$.
    Analogously, write $w' = w'^{(1)} \circ w'^{(2)} \circ w'^{(3)}$, where each $w'^{(r)}$ one third of $w'$.
    Then, since $w' = w + e_i$ we have
    \begin{eqnarray*}
        \E[\dist(f_j(w), w) | E_4, \mineq = j]
        &=& \E[\dist(f_{j-1}(w^{(1)}), w^{(1)}) | \recmaj[j-1](w^{(1)}) = 0, \recmaj[j-1](w^{(1)}+e_i) = 1] \\
        && + \E[\dist(f_{j-1}(w^{(2)}), w^{(2)}) | \recmaj[j-1](w^{(2)}) = 0] \\
        & \leq &  2 \cdot (1.5^{j-1} - 1) + 1.5^{j-1} = 3 \cdot 1.5^{j-1} - 2,
    \end{eqnarray*}
    where the last inequality is by \cref{eq:dist-1.5^k} and \cref{eq:dist-cond-on-boundary-1.5^k}.
    Similarly,
    \begin{equation*}
        \E[\dist(f_j(w'), w') | E_4, \mineq = j]
        = \E[\dist(f_{j-1}(w'^{(3)}), w'^{(3)}) | \recmaj[j-1](w'^{(3)}) = 0]
        \leq 1.5^{j-1},
    \end{equation*}
    where the last inequality is by \cref{eq:dist-1.5^k}.
    Therefore,
    \begin{equation*}
        \E[\dist(f_\mineq(w), f_\mineq(w')) | E_4, \mineq = j] < 4 \cdot 1.5^{j-1}.
    \end{equation*}
    The cases of $y = 010$ and $001$ are handled similarly, and it is straightforward to verify that in these cases
    we also get the bound of $4 \cdot 1.5^{j-1}$.

    By combining \cref{eq:mineq-distr} with \cref{eq:dist-cond-mineq} it follows that
    \begin{eqnarray*}
      \E[\dist(f_k(x), f_k(x+e_i)) | E_4] \cdot \Pr[E_4]
      &=& \sum_{j=1}^{k} \E[\dist(f_k(x), f_k(x+e_i)) | E_4, \mineq = j] \cdot \Pr[\mineq = j | E_4] \cdot \Pr[E_4]\\
      & \leq & \sum_{j=1}^{k} 4 \cdot 1.5^{j-1} \cdot \Pr[\mineq = j] \\
      & \leq & 4 \cdot \sum_{j=1}^{k} 1.5^{j-1} \cdot 2^{-j} \leq 8.
    \end{eqnarray*}
    This completes the proof of \cref{claim:recmaj-E4}.
  \end{proof}

\section{Average stretch for tribes}
\label{sec:tribes}

In this section we prove \cref{thm:tribes-avgstretch}, showing a mapping from $\HC_n$ to $\Astribes$ with $O(\log(n))$ average stretch.
Let $\mu_{\tribes}^1$ be the uniform distribution on $\Atribes$, and let $\mu_{\tribes}^0$  be the uniform distribution on $\Ztribes = \HC_n \setminus \Atribes$.
The proof consists of the following two claims.
\begin{claim}\label{claim:tribes-0-1-mapping}
    For $\mu_{\tribes}^1$ and $\mu_{\tribes}^0$ as above it holds that
    \begin{equation*}
        W_1(\mu_{\tribes}^0, \mu_{\tribes}^1) = O(\log(n)).
    \end{equation*}
\end{claim}

Next, let $\Astribes \seq \HC_n$ be an arbitrary superset of $\Atribes$ of density $1/2$,
and let $\mu_{\tribes}^*$ be the uniform distribution on $\Astribes$.
\begin{claim}\label{claim:tribes*-mapping}
    Consider $\HC_{n-1}$ as $\{ x \in \HC_n : x_n = 0\}$, and let $\mu_{n-1}$ be the uniform measure on $\HC_{n-1}$.
    Then,
    \begin{equation*}
        W_1(\mu_{n-1}, \mu_{\tribes}^*) \leq W_1(\mu_{\tribes}^0, \mu_{\tribes}^1) + O(\log(n)).
    \end{equation*}
\end{claim}

By combining \cref{claim:tribes-0-1-mapping} and \cref{claim:tribes*-mapping}
we get that the average transportation distance between $\HC_{n-1}$ and $\Astribes$ is $W_1(\mu_{n-1}, \mu_{\tribes}^*) = O(\log(n))$.
By \cref{claim:W_1-bijection} it follows that there exists $\phi_{\tribes} \colon \HC_{n-1} \to \Astribes$
such that $\E[\dist(x,\phi_{\tribes}(x))] = O(\log(n))$,
and using \cref{prop:avg-dist-avg-stretch} we conclude that
$\avgstretch(\phi_{\tribes}) = O(\log(n))$.
This completes the proof of \cref{thm:tribes-avgstretch}. \qed

Below we prove \cref{claim:tribes-0-1-mapping} and \cref{claim:tribes*-mapping}.

\begin{proof}[Proof of \cref{claim:tribes-0-1-mapping}]
    Let $\D = \D_w$ the uniform distribution over $\Bits^{w} \setminus \{1^w\}$,
    let $p=2^{-w}$, and denote by $\L = \L_{w,s}$ the binomial distribution $\Bin(p, s)$ conditioned on the outcome being positive.
    That is,
    \begin{equation*}
        \Pr[\L = \ell] = \frac{ {\binom{s}{\ell}}p^\ell (1-p)^{s-\ell}}{\sum_{j=1}^s  {\binom{s}{j}}p^j (1-p)^{s-j}}
        \quad \forall \ell \in \{1,\dots,s\}.
    \end{equation*}

    Note that $\mu_{\tribes}^0$ is equal to the product distribution $\D^s$.
    Note also that in order to sample from the distribution $\mu_{\tribes}^1$, we can first sample $\L \in \{1,\dots,s\}$,
    then choose $\L$ random tribes that vote unanimously 1, and for the remaining $s - \L$ tribes
    sample their values in this tribe according to $\D$.

    We define a coupling $q_\tribes$ between $\mu_{\tribes}^0$ and $\mu_{\tribes}^1$ as follows.
    First sample $x$ according to $\mu_{\tribes}^0$.
    Then, sample $\L \in \{1,\dots,s\}$, choose $\L$ tribes $T \seq [s]$ uniformly,
    and let $S = \{ (t-1)w + j : t \in T, j \in [w]\}$ be all the coordinates participating in all tribes in $T$.
    Define $y \in \HC_n$ as $y_i = 1$ for all $i \in S$, and $y_i = x_i$ for all $i \in [n] \setminus S$.
    It is clear that $y$ is distributed according to $\mu_{\tribes}^1$,
    and hence $q_\tribes$ is indeed a coupling between $\mu_{\tribes}^0$ and $\mu_{\tribes}^1$.

    We next show that
    $\E_{(x,y) \sim q_\tribes}[\dist(x,y)]  = O(\log(n))$.
    We have $\E_{(x,y) \sim q_\tribes}[\dist(x,y)] \leq \E[\L \cdot w]$,
    and by the choice of parameters, we have $w \leq \log(n)$
    and $\E[\L] 
    = \frac{\E[\Bin(2^{-w},s)]}{1 - \Pr[\Bin(2^{-w},s) = 0]} = \frac{s \cdot 2^{-w}}{1 - 2^{-ws}}$.
    By the choice of $s \leq \ln(2) 2^w + O(1)$ it follows that $\E[\L] = O(1)$,
    and hence
    \begin{equation*}
        W_1(\mu_{\tribes}^0, \mu_{\tribes}^1) \leq \E_{(x,y) \sim q_\tribes}[\dist(x,y)] \leq \E[\L \cdot w] = O(\log(n)).
    \end{equation*}
    This completes the proof of \cref{claim:tribes-0-1-mapping}.
\end{proof}

\begin{proof}[Proof of \cref{claim:tribes*-mapping}]
We start by showing that
\begin{equation}\label{eq:q-tribes-Hn}
    W_1(\mu_n, \mu_{\tribes}^1) \leq W_1(\mu_{\tribes}^0, \mu_{\tribes}^1),
\end{equation}
where $\mu_n$ is the uniform measure on $\HC_n$.
Indeed, let $q_\tribes$ be a coupling between $\mu_{\tribes}^0$ and $\mu_{\tribes}^1$.
Define a coupling $q_n$ between $\mu_n$ and $\mu_{\tribes}^1$ as
\begin{equation*}
    q_n(x,y) = \begin{cases}
               \frac{|\Ztribes|}{2^n} \cdot q_{\tribes}(x,y) & \mbox{if } x \in \Ztribes \mbox{ and } y \in \Atribes, \\
               1/2^n & \mbox{if } x = y \in \Atribes, \\
               0 & \mbox{otherwise}.
             \end{cases}
\end{equation*}
It is straightforward to verify that $q_n$ is indeed a coupling between $\mu_n$ and $\mu_{\tribes}^1$.
Letting $q_\tribes$ be a coupling for which $\E_{(x,y) \sim q_\tribes}[\dist(x,y)] = W_1(\mu_{\tribes}^0, \mu_{\tribes}^1)$ we get
\begin{eqnarray*}
    W_1(\mu_n, \mu_{\tribes}^1)
    & \leq & \sum_{\substack{x \in \HC_{n}\\ y \in \Atribes}} \dist(x,y) q_n(x,y) \\
    & = &
    \sum_{\substack{x \in \Ztribes \\ y \in \Atribes}} \dist(x,y) q_n(x,y) +
    \sum_{\substack{x \in \Atribes \\ y \in \Atribes}} \dist(x,y) q_n(x,y) \\
    & = & \frac{|\Ztribes|}{2^n} \E_{(x,y) \sim q_\tribes}[\dist(x,y)] + \sum_{x \in \Atribes} \dist(x,x) q_n(x,x)\\
    & = & \frac{|\Ztribes|}{2^n} \cdot W_1(\mu_{\tribes}^0, \mu_{\tribes}^1) < W_1(\mu_{\tribes}^0, \mu_{\tribes}^1),
\end{eqnarray*}
which proves \cref{eq:q-tribes-Hn}.

Next, we show that
\begin{equation}\label{eq:q-tribes-Hn-1}
    W_1(\mu_{n-1}, \mu_{\tribes}^1) \leq  W_1(\mu_n, \mu_{\tribes}^1) + 1.
\end{equation}
Indeed, let $q_n$ be a coupling between $\mu_n$ and $\mu_{\tribes}^1$
minimizing $\sum_{(x,y) \in \HC_{n} \times \Atribes} \dist(x,y) q_n(x,y)$.
Define a coupling $q_{n-1}$ between $\mu_{n-1}$ and $\mu_{\tribes}^1$
as
\begin{equation*}
    q_{n-1}(x,y) =
               q_n(x,y) + q_n(x+e_n,y)
               \quad \forall x \in \HC_{n-1} \mbox{ and } y \in \Atribes.
\end{equation*}
It is clear that $q_{n-1}$ is a coupling between $\mu_{n-1}$ and $\mu_{\tribes}^1$.
Next we prove \cref{eq:q-tribes-Hn-1}.
\begin{eqnarray*}
W_1(\mu_{n-1}, \mu_{\tribes}^1)
&\leq& \sum_{\substack{x \in \HC_{n-1}\\ y \in \Atribes}} \dist(x,y) q_{n-1}(x,y) \\
&=& \sum_{\substack{x \in \HC_{n-1}\\ y \in \Atribes}} \dist(x,y) q_{n}(x,y)
    + \sum_{\substack{x \in \HC_{n-1}\\ y \in \Atribes}} \dist(x,y) q_{n}(x+e_i,y) \\
&\leq & \sum_{\substack{x \in \HC_{n-1}\\ y \in \Atribes}} \dist(x,y) q_{n}(x,y)
    + \sum_{\substack{x \in \HC_{n-1}\\ y \in \Atribes}} (\dist(x+e_i,y) +1) q_{n}(x+e_i,y) \\
&=& \sum_{\substack{x \in \HC_{n}\\ y \in \Atribes}} \dist(x,y) q_n(x,y)  + \sum_{\substack{x \in \HC_{n-1}\\ y \in \Atribes}} q_{n}(x+e_i,y)\\
&\leq & W_1(\mu_{n}, \mu_{\tribes}^1) + 1,
\end{eqnarray*}
which proves \cref{eq:q-tribes-Hn-1}.

Next, we show that
\begin{equation}\label{eq:q-tribes*}
    W_1(\mu_{n-1}, \mu_{\tribes}^*) \leq W_1(\mu_{n-1}, \mu_{\tribes}^{1}) + O(\log(n)).
\end{equation}
In order to prove \cref{eq:q-tribes*}, let $\delta = \half - \frac{\abs{\Atribes}}{2^n}$.
By the discussion in \cref{sec:tribes-def} we have $\delta = O\left(\frac{\log(n)}{n}\right)$.
Then $\abs{\Astribes \setminus \Atribes} = \delta \cdot 2^n$.
Let $q_{n-1}$ be a coupling between $\mu_{n-1}$ and $\mu_{\tribes}^1$ such that $\E_{(x,y) \sim q_{n-1}}[\dist(x,y)] = W_1(\mu_{n-1}, \mu_{\tribes}^{1})$.
Define a coupling $q^*$ between $\mu_{n-1}$ and $\mu_{\tribes}^*$ as
\begin{equation*}
    q^*(x,y) = \begin{cases}
               (1-2\delta) \cdot q_{n-1}(x,y) & \mbox{if } x \in \HC_{n-1} \mbox{ and } y \in \Atribes, \\
               4 \cdot 2^{-2n} & \mbox{if } x \in \HC_{n-1} \mbox{ and } y \in \Astribes \setminus \Atribes. \\
             \end{cases}
\end{equation*}
It is straightforward to verify that $q^*$ is a coupling between $\mu_{n-1}$ and $\mu_{\tribes}^*$.
Next we prove \cref{eq:q-tribes*}.
\begin{eqnarray*}
    W_1(\mu_{n-1}, \mu_{\tribes}^*)
    & \leq & \sum_{\substack{x \in \HC_{n-1}\\ y \in \Astribes}} \dist(x,y) \cdot q^*(x,y) \\
    & = & (1-2\delta) \sum_{\substack{x \in \HC_{n-1}\\ y \in \Atribes}} \dist(x,y) q_{n-1}(x,y)
    +  \sum_{\substack{x \in \HC_{n-1}\\ y \in \Astribes \setminus \Atribes}} \dist(x,y) \cdot 4 \cdot 2^{-2n} \\
    & \leq & (1 - 2\delta) \cdot W_1(\mu_{n-1}, \mu_{\tribes}^{1}) + 2 \delta \cdot \max_{\substack{x \in \HC_{n-1}\\y \in \HC_{n}}}(\dist(x,y)).
\end{eqnarray*}
\cref{eq:q-tribes*} follows from the fact that $\max(\dist(x,y)) \leq n$ and $\delta = O\left(\frac{\log(n)}{n}\right)$.

By combining \cref{eq:q-tribes-Hn,eq:q-tribes-Hn-1,eq:q-tribes*} we get
$W_1(\mu_{n-1}, \mu_{\tribes}^*) \leq W_1(\mu_{\tribes}^0, \mu_{\tribes}^1) + O(\log(n))$.
\end{proof}

\section{Concluding remarks and open problems}
\label{sec:open-problems}

\paragraph{Uniform upper bound on the average stretch.}
We've shown a uniform upper bound of $O(\sqrt{n})$ on the average transportation distance $\E[\dist(x,\phi(x))]$
from $\HC_{n-1}$ to any set $A \seq \HC_n$ of density $1/2$, where $\HC_{n-1}$ is treated as $\{ x \in \HC_n : x_n = 0\}$.
This bound is tight up to a multiplicative constant.
Indeed, it is not difficult to see that
for any bijection $\phi$ from $\HC_{n-1}$ to $A_{\maj} = \{x \in \HC_n : \sum_i x_i > n/2\}$ (for odd $n$)
the average transportation of $\phi$ is $\E[\dist(x,\phi(x))] \geq \Omega(\sqrt{n})$.

In contrast, we believe that the upper bound of $O(\sqrt{n})$ on the average stretch is not tight, and it should be possible to improve it further.
\begin{problem}
    Prove/disprove that for any set $A \seq \HC_n$ of density $1/2$
    there exists a mapping $\phi_A \colon \HC_{n-1} \to A$ with $\avgstretch(\phi) = o(\sqrt{n})$.
\end{problem}

\paragraph{The tribes function.}
Considering our results about the tribes function, we make the following conjecture.
\begin{conjecture}
    Let $w$ be a positive integer, and let $s$ be the largest integer such that $1 - (1-2^{-w})^s \leq 1/2$.
    For $n=s \cdot w$ let $\tribes \colon \HC_n \to \Bits$ be defined as a DNF consisting of $s$ disjoint clauses of width $w$,
    and let $\Atribes = \{x \in \HC_n : \tribes(x) = 1\}$.
    There exists $\Astribes \seq \HC_n$ a   superset of $\Atribes$ of density $\mu_n(\Astribes) = 1/2$
    such that $W_1(\mu_{n-1}, \mu_{\tribes}^*) = O(1)$,
    where $\mu_{\tribes}^*$ is the uniform distribution on $\Astribes$.
\end{conjecture}

As a first step toward the conjecture we propose the following strengthening of \cref{claim:tribes-0-1-mapping}.
\begin{problem}\label{problem:tribes-W1}
    Let $\mu_{\tribes}^1$ be the uniform distribution on $\Atribes$, and let $\mu_{\tribes}^0$  be the uniform distribution on $\Ztribes = \HC_n \setminus \Atribes$.
    It is true that $W_1(\mu_{\tribes}^0, \mu_{\tribes}^1) = O(1)$?
\end{problem}

\paragraph{A candidate set that requires large average stretch.}
We propose a candidate set $A^*$ for which we hope that any mapping from $\HC_{n-1}$ to $A^*$ requires large average stretch.
The set is defined as follows. Let $k^* \in [n]$ be the maximal $k$
such that $\binom{n}{\leq k} = \sum_{j=0}^k {\binom{n}{j}} \leq 2^{n-2}$.
Let $B^0_{1/4} =\{ x \in \HC_n : \sum_{i \in [n]} x_i \leq k\}$
and $B^1_{1/4} =\{ x \in \HC_n : \sum_{i \in [n]} x_i \geq n-k\}$
be two (disjoint) antipodal balls of radius $k^*$,
and let $C \seq \HC_n \setminus (B^0_{1/4} \cup B^1_{1/4})$ be an arbitrary set of size
$\abs{C} = 2^{n-1} - \abs{B^0_{1/4} \cup B^1_{1/4}}$.
Define $A^* = B^0_{1/4} \cup B^1_{1/4} \cup C$.

\begin{conjecture}
    There is no bijection $\phi^* \colon \HC_{n-1} \to A^*$ with $\avgstretch(\phi^*) = O(1)$.
\end{conjecture}


\section*{Acknowledgements}

We are grateful to the anonymous referees for carefully reading the paper,
and for their very helpful comments that greatly improved the presentation.

\bibliographystyle{alpha}
\bibliography{references}
\end{document}